\documentclass{amsart}
\usepackage{graphicx}
\usepackage{amsthm}
\usepackage{amsmath}
\usepackage{amssymb}
\usepackage{color}
\newtheorem{thm}{Theorem}[section]

\theoremstyle{definition}

\theoremstyle{remark}
\newtheorem{remark}[thm]{Remark}

\newtheorem{theorem}{Theorem}
\newtheorem{lemma}{Lemma}
\newtheorem{proposition}{Proposition}
\newtheorem{definition}{Definition}
\newtheorem{corollary}{Corollary}

\newtheorem{example}{Example}

\newcommand{\RR}{\mathbb R}

\newcommand{\Iff}{if and only if}

\numberwithin{equation}{section}
\newcommand{\pc}[1]{{\color{magenta}{#1}}}

\begin{document}

\title[A note on phase (norm) retrievable
Real Hilbert (fusion) frames]{}
\author{F. Akrami}
\address{Department of sciences, University of Maragheh, Maragheh, Iran.}
\email{fateh.akrami@gmail.com}
\author{P. G. Casazza}
\address{Department of Mathematics, University of Missouri, Columbia, USA.}
\email{casazzap@missoouri.edu}
\author{M. A. Hasankhani Fard}
\address{Department of Mathematics, Vali-e-Asr University, Rafsanjan, Iran.}
\email{m.hasankhani@vru.ac.ir}
\author{A. Rahimi}
\address{Department of sciences, University of Maragheh, Maragheh, Iran.}
\email{rahimi@maragheh.ac.ir}

\dedicatory{}

\thanks{The second author was supported by NSF DMS 1609760}



%\date{}%
%\dedicatory{}%
%\commby{}%
% ----------------------------------------------------------------
\begin{abstract}
In this manuscript, we present several new results in finite and countable dimensional real Hilbert space phase retrieval and norm retrieval by vectors and projections. We make a detailed study of when hyperplanes do norm retrieval.
Also, we show that the families of norm retrievable frames $\{f_{i}\}_{i=1}^{m}$ in $\mathbb{R}^n$ are not dense in the family of $m\leq (2n-2)$-element sets of
vectors in $\mathbb{R}^n$ for every finite $n$ and the families of vectors which do norm retrieval in $\ell^2$ are not dense in the infinite families of vectors in $\ell^2$. We also show that if a Riesz basis does norm retrieval in $\ell^2$, then it is an orthogonal sequence. We provide numerous examples to show that our results are best possible.
\end{abstract}

\maketitle

\section{Introduction}\label{s:intro}
The concept of frames in a separable Hilbert space was originally introduced by Duffin and Schaeffer in the context of non-harmonic Fourier series \cite{DS1952}.
Frames have the redundancy property that make them more applicable than bases. Phase retrieval and norm retrieval are one of the most applied and studied areas of research today. Phase retrieval for Hilbert
space frames was introduced in \cite{BCE2006} and quickly became an industry. Although much work has been done on the complex infinite dimensional case of phase retrieval, only a few papers exist on
infinite dimensional real phase retrieval or norm retrieval, e.g., \cite{BCCHT18}. In \cite{BCCHT18},  some concepts such as ``full spark" and ``finitely full spark" were introduced and were generalized. We
will present some examples for them.
\par
Fusion frames are an emerging topic of frame theory, with applications to communications and distributed
processing. Fusion frames were introduced by Casazza and Kutyniok in \cite{CK2004} and further developed in their joint paper \cite{CKL2008} with Li. The theory for fusion frames
is available in arbitrary separable Hilbert spaces (finite dimensional or not).
\par
We first give the background material needed for the paper. Let $\mathbb{H}$ be finite or infinite dimensional Real Hilbert space and B$(\mathbb{H})$ be the class of all bounded linear operators defined on $\mathbb{H}$. The natural numbers and real numbers are denoted by $``\mathbb{N}"$ and $``\mathbb{R}"$, respectively. We  use $[m]$ instead of the set $\{1,2,3,\dots,m \}$ and use $[\{f_i\}_{i \in I}]$ instead of $span\{f_i\}_{i \in I}$ where $I$ is a finite or countable subset of $\mathbb{N}$. We denote by $\mathbb{R}^n$ a $n$ dimensional real  Hilbert space.
 We start with the definition of a real Hilbert space frame.
\begin{definition} \label{D:frame}
 A family of vectors $\{f_i\}_{i\in I}$ in a finite or infinite dimensional separable Hilbert space $\mathbb {H}$ is a \textbf{frame} if there are constants $0<A \leq B< \infty $ so that
$$ A\|f\|^2 \leq \sum_{i\in I}{|\langle f,f_i \rangle|^2}\leq B\|f\|^2, \quad \mbox{for all} \quad f\in \mathbb{H}.$$
The constants $A$  and $B$ are called the lower and upper frame bounds for $\{f_i\}_{i\in I}$, respectively. If only an upper frame bound exists, then $\{f_i\}_{i\in I}$ is called a {\bf B-Bessel set} or simply
{\bf Bessel} when the constant is implicit. If $A=B$, it is called an {\bf A-tight frame} and
 in case $ A=B=1$, it called a {\bf Parseval~frame}.  The values $\{\langle f, f_i \rangle\}_{i=1}^{\infty}$ are called the frame coefficients of the vector $f \in \mathbb{H}$.
\end{definition}
We will need to work with Riesz sequences.
\begin{definition} \label{D:Riesz sequence}
A family $\Phi = \{\phi_i\}_{i \in I}$ in a finite or infinite dimensional Hilbert space $\mathbb{H}$
is a \textbf{Riesz sequence} if there are constants $0<A \leq B< \infty $ satisfying
$$A\sum_{i \in I}|c_i|^2 \leq \|\sum_{i\in I}{c_i \phi_i}\|^2\leq B\sum_{i \in I}|c_i|^2$$
 for all sequences of scalars $\{c_i\}_{i \in I}$.
If it is complete in $\mathbb{H}$, we call $\Phi$ a \textbf{Riesz basis}.
\end{definition}

 It is well known that every finite dimensional real Hilbert space  $\mathbb{H}$ \textrm is isomorphic to $\mathbb{R}^n$, for some n, and every infinite
dimensional real Hilbert space  $\mathbb{H}$ \textrm is  isomorphic to $\ell^2(\mathbb{R})$ (countable real sequences with $\ell^2$~-norm). We will use $\ell^2$ instead of $\ell^2(\mathbb{R})$ for simplicity.
Throughout the paper, $\{e_i\}_{i=1}^{\infty}$ will be used to denote the canonical basis for the real space $\ell^2$, i.e., a basis for which
$$\langle e_i,e_j \rangle=\delta_{i,j}=
\begin{cases}
1 \quad  if \ i=j,  \\
0 \quad  if \ i\ne j.
\end{cases}$$

\begin{definition}\label{D:phase(norm) retrieval by vectors}
\textrm A family of vectors $\{f_{i}\}_{i\in I}$ in a \pc{real} Hilbert space $\mathbb {H}$ does \textbf {phase (norm) retrieval} if whenever $ x, y \in \mathbb {H}$, satisfy
$$|\langle x,f_{i}\rangle |=|\langle y,f_{i}\rangle| \quad \mbox{ for all } i\in I,$$
then $x=\pm y \;( \|x\|=\|y\|)$. \\ [10pt]
\end{definition}
Note that if $\{f_i\}_{i\in I}$ does phase (norm) retrieval, then so does $\{a_if_i\}_{i\in I}$ for any $0< a_i< \infty$ for all $i\in I$.
But in the case where $|I|=\infty$, we have to be careful to maintain frame bounds. This always works if $0<\inf_{i\in I}a_i \le sup_{i\in I}a_i < \infty$.
But this is not necessary in general.

 The complement property is an essential issue here.
  Since in the finite dimensional setting frames are equivalent with spanning sets, first we give the complement property in the finite case from \cite{CC17}.
\begin{definition}\label{D:complement property in finite case}
\textrm A family of vectors $\{f_{k}\}_{k=1}^{m}$ in $\mathbb{R}^n $ \textrm has the \textbf{complement property}\, \textrm  if for any subset $I \subset [m]$,
\center $either\ \  span\{f_{k}\}_{k\in I}=\mathbb{R}^n$ \textrm \quad or\ \ \quad  $ span\{f_{k}\}_{k\in I^c}=\mathbb{R}^n $.\\ [10pt]
\end{definition}
This is generalized in \cite{CCD16}.
\begin{definition}\label{D:complement pro in infinite case}
\textrm A family of vectors $\{f_{k}\}_{k=1}^{\infty}$ in $\ell^2 $ \textrm has the \textbf{complement property}\, \textrm  if for any subset $I \subset \mathbb{N}$, \\
$$either\ \  \overline{span}\{f_{k}\}_{k\in I}=\ell ^2 \textrm \quad or \quad \ \  \overline{span}\{f_{k}\}_{k\in I^c}=\ell^2. $$
\end{definition}
The following result appeared in \cite{CCD16}.
\begin{theorem}\label{T:phase retrievality and complement property for infinite case}
 A family of vectors $\{f_{i}\}_{i=1}^{\infty} $  \textrm does phase retrieval for $\ell^2$ if and only if  it has the complement property.
\end{theorem}
The corresponding finite dimensional result first appeared in \cite{CCPW16}.
\begin{theorem}\label{T:phase retrievality and complement property for finite case}
A family of vectors $\{f_{i}\}_{i=1}^{m} $ in $\mathbb {R}^n$  \textrm does phase retrieval if and only if it has the complement property.
\end{theorem}

We recall

\begin{definition}
A family of vectors $\{f_i\}_{i=1}^m$ in $\RR^n$ is {\bf full spark} if for every $I\subset [m]$ with $|I|=n$, the set $\{f_i\}_{i\in I}$
spans $\RR^n$.
\end{definition}
\begin{corollary}\label{T:phase retrievality and complement property and full apark for finite case}
\textrm If $\{f_{i}\}_{i=1}^{m} $  \textrm does phase retrieval in $\mathbb {R}^n$, then $m \geq2n-1$.
If $m\ge 2n-1$ and the frame is full spark, then it does phase retrieval.
If $m=2n-1$, $\{f_{i}\}_{i=1}^{m} $ does phase retrieval if and only if it is full spark.
\end{corollary}

For linearly independent sets there is a special case \cite{CCJW14}.
\begin{theorem}\label{TT}
If $\{f_i\}_{i=1}^n$ in $\RR^n$ does norm retrieval, then the set is orthogonal.
\end{theorem}

It is clear that phase retrieval implies norm retrieval. The converse fails since an orthonormal basis does norm retrieval but fails phase
retrieval since it fails complement property.
Subsets of phase (norm) retrievable frames certainly may fail phase (norm) retrieval, since linearly independent subsets fail the complement property so
fail phase retrieval and
by Theorem \ref{TT} if every subset of a frame does norm retrieval then every two distinct vectors are orthogonal and so the frame is an orthogonal set plus possibly
more vectors.
 But
projections of these sets do still do phase (norm) retrieval.

\section{Phase (norm) retrievable Fusion frames}\label{s:phase(norm) retrievable Fusion frames}

In real life and some areas such as crystal twinning in X-ray crystallography \cite{D2010}, we need to project the signal onto higher than one
dimensional subspaces and it has to be recovered from the norms of these
projections. Throughout the paper, the term projection is used to describe an orthogonal projection onto subspaces. Norm retrieval is in fact the essential condition to pass phase retrievability of these
projections to the corresponding orthogonal complements \cite{CCJW14}.

Fusion frames can be regarded as a generalization of conventional frame theory.
It turns out that the fusion frame theory is in fact more delicate due to complicated
relationships between the structure of the sequence of weighted subspaces and the
local frames in the subspaces and due to sensitivity with respect to changes of the
weights. Fusion frames were introduced by Casazza and Kutyniok in \cite{CK2004}( under the name\textbf{ frames of subspaces}) and further developed in their joint paper \cite{CKL2008} with Li.
Here $\{W_{i}\}_{i \in I}$ is a family of closed subspaces of $\mathbb {H}$ and $\{v_{i}\}_{i \in I}$ is a family of positive weights. Also we denote by $P_{i}$ the orthogonal projection onto
$W_i$.
\begin{definition}\label{D: fusion frame}
A family $\{(W_{i},v_{i}) \}_{i \in I }$ with $W_i$ subspaces of $\mathbb{H}$, $v_i$ weights, and $P_i$ the
projection onto $W_i$, is a\textbf{ fusion frame} for $\mathbb{H}$  if  there exist constants
$A,B > 0$ such that
$$A\|f\|^2 \leq \sum_{i \in I}v_{i}^{2}\|P_i f\|^2 \leq B\|f\|^2, \mbox{ for all } f \in \mathbb{H}.$$ The constants $A$ and $B$ are called the
\textbf{fusion frames bounds}.  We also refer to the fusion frames as $\{P_i,v_i\}_{i\in I}$ or just $\{P_i\}_{i\in I}$ if the weights are all one.
\end{definition}
For more details on fusion frames, we recommend \cite{CK2004}. Improving and extending the notions of phase  and norm  retrievability, we present the definition of phase (norm) retrievable to fusion frames.

\begin{definition}\label{D:phase(norm) retrieval by vectors}
\textrm A family of projections $\{P_{i}\}_{i\in I}$ in a real Hilbert space $\mathbb {H}$ does \textbf {phase (norm) retrieval} if whenever $ x, y \in \mathbb {H}$, satisfy
$$\|P_ix\|=\|P_iy\| \quad \mbox{ for all } i\in I,$$
then $x=\pm y \;( \|x\|=\|y\|)$. \\ [10pt]
\end{definition}

\begin{definition}\label{D:phase(norm) retrievable fusion frame}
\textrm {A fusion frame} $\{(W_{i},v_{i}) \}_{i \in I }$ \, \textrm{is phase (norm) retrievable for} $\mathbb{H}$ \, \Iff \, the family of projections $\{P_{i} \}_{i\in I}$ is phase (norm) retrievable for
$\mathbb{H}$, where $P_{i}=P_{W_{i}}$ is the orthogonal projection onto $W_{i}$ \, $(i \in I)$.
\end{definition}

\begin{remark}
Note that $\{W_i,v_i\}_{i=1}^m$ does phase (norm) retrieval if and only if $\{W_i\}_{i=1}^m=:\{W_i,1\}_{i=1}^m$ does phase (norm) retrieval.
\end{remark}

It is well known that phase (norm) retrievable sets need not be frames (fusion frames). For example consider
phase (norm) retrievable set $\{e_i + e_j\}_{i<j}$ which does not satisfy the frame upper bound condition and therefore is not a (fusion) frame for $\ell^2$, but it
does phase retrieval.

Part of the importance of fusion frames is that it is both necessary and sufficient  to be able to string together frames for each of the subspaces $W_{k}$ (with uniformly bounded frame constants) to get a frame for $\mathbb{H}$ which is proved in  \cite{CCPW16}:
\begin{theorem} \label{main}
Let $\{W_i\}_{i\in I}$ be subspaces of $R^n$.  The following are equivalent:
\begin{enumerate}
\item $\{W_i\}_{i\in I}$ is phase retrievable.
\item For every orthonormal basis $\{f_{ij}\}_{j\in I_i}$ for $W_i$, the family $\{f_{ij}\}_{j\in I_i,i\in I}$ does phase retrieval.
\end{enumerate}
\end{theorem}

We note that (2) of the theorem must hold for {\bf every} orthonormal basis for the subspaces.
For example, let $\{\phi_i\}_{i=1}^3$ and $\{\psi_i\}_{i=1}^3$ be orthonormal bases for $R^3$ so that $\{\phi_i\}_{i=1}^3\cup \{\psi_i\}_{i=1}^3$
is full spark. Let
\[ W_1=[\phi_1]\ \ W_2=[\phi_2]\  \ W_3=[\phi_3]\  \ W_4=[\psi_1,\psi_2].\]
Then $\{W_i\}_{i=1}^4$ is a fusion frame for $R^3$ and $\{\phi_1,\phi_2,\phi_3,\psi_1,\psi_2\}$ is full spark and so does phase retrieval
for $R^3$. But it is known that 4 subspaces of $R^3$ cannot do phase retrieval \cite{CCPW16}.

The corresponding result for norm retrieval does not make sense, because every orthonormal basis for a subspace does norm retrieval.

We can strengthen this theorem.
\begin{theorem}\label{T:phase(norm) retrievable Fusion frames}
Let $\{(W_{k},v_{k}) \}_{k \in I }$ be a phase (norm) retrievable fusion frame for $\mathbb{H}$  and $\{f_{ij} \}_{j \in {I_i}}$ be a norm retrievable frame for $W_{i}$ for $i\in I$. Then $\{v_{i}f_{ij} \}_{j \in {I_i}, i \in I}$ is a phase (norm) retrievable frame for $\mathbb {H}$.
\end{theorem}
\begin{proof}
 Let $f,g \in \mathbb {H}$, For any  $j \in I_{i},i \in I$, we have
\begin{align}
|\langle f, v_{i}f_{ij}\rangle|=|\langle g, v_{i}f_{ij}\rangle| \: \notag
&\Rightarrow \;  |\langle f, v_{i}P_{i}f_{ij}\rangle|=|\langle g, v_{i}P_{i}f_{ij}\rangle|,\:\\ \notag
&\Rightarrow \;  v_i|\langle P_{i} f, f_{ij}\rangle|=v_i|\langle P_{i}g, f_{ij}\rangle|,\:  \\ \notag
\mbox{since } \{f_{ij} \}\: \textrm{do norm retrieval} \:
&\Rightarrow \: \|P_{i}f\|=\|P_{i}g\|, \: \forall {i \in I},  \\ \notag
\mbox{since} \:\{P_{i}\}_{i \in I}\: \textrm {do phase retrieval} \:
&\Rightarrow \: f=\pm g.
\end{align}
\end{proof}
The following theorem shows that the unitary operators preserve  phase (norm) retrievability of fusion frames.
\begin{theorem}\label{T:unitary operator and norm retrievability}
Let $\{(W_i,v_i ) \}_{i \in I}$ be a phase (norm) retrievable fusion frame for $\mathbb{H}$. If $T \in B(\mathbb{H})$
is a unitary operator, then
$\{(TW_i,v_i ) \}_{i \in I}$ is also a phase (norm) retrievable fusion frame.
\end{theorem}
\begin{proof}
Let $\{P_i\}_{i\in I}$ be the projections onto $\{W_i\}_{i\in I}$. The projections onto $TW_i$ are $\{Q_i=TP_iT^*\}_{i\in I}$.
Assume $f,g\in \mathbb{H}$ and
\[ \|Q_if\|=\|Q_ig\|, \mbox{ for all }i\in I.\]
Then,
\[ \|TP_iT^*f\|=\|P_iT^*f\|=\|TP_iT^*g\|=\|P_iT^*g\|,\mbox{ for all }i\in I.\]
If $\{P_i\}_{i\in I}$ does norm retrieval, then $\|T^*f\|=\|T^*g\|$, and so $\|f\|=\|g\|$. If $\{P_i\}_{i\in I}$ does phase retrieval, then
$T^*f=\pm T^*g$, and so $f=\pm g$.
  \end{proof}

 \begin{theorem}\label{T:union of norm retrievable Fusion frames}
Let $\{(W_{i},v_{i}) \}_{i \in I}$ be a norm retrievable fusion
frame for a Hilbert space $\mathbb{H}$, with projections $\{P_i\}_{i\in I}$.  Let $\{Q_i\}_{i\in I}$ be projections from
$W_i$ to $W_i$, and let $W_i'=Q_iW_i$ and $W_i''=(I-Q_i)W_i$ for all $i\in I$. Then
$\{(W_{i}',v_{i}) \}_{i \in {I}}\bigcup \{( W_{i}'',v_{i}) \}_{i \in {I}}$  is a
 norm retrievable fusion frame for $\mathbb{H}$.
 \end{theorem}
\begin{proof}
 If $\|Q_{i} x\|=\|Q_{i} y\|$ and $\|(I-Q_i) x\|=\|(I-Q_i)y\|$, for all $i\in I$, then
   \begin{align}
  \|P_ix\|^2&=\|Q_ix\|^2+\|(I-Q_i)x\|^2  \\ \notag
  &=\|Q_{i}y\|^2 + \|(I-Q_{i})y\|^2  \\ \notag
  &=\|P_{i}y\|^2  \end{align}
Since $\{P_i\}_{i\in I}$ does norm retrieval, we have $\|x\|=\|y\|$.
\end{proof}

The characterization of norm retrievable families of vectors first appeared in \cite{MAHF2015}.

\begin{theorem}\label{T:norm retrievablity and perp}
 A family of vectors $\{f_{k}\}_{k=1}^{\infty} $ \textrm does norm retrieval for $\mathbb {H}$ if and only if
 for any subset $I \subset \mathbb{N},$
 $$  (\overline{span}\{f_{k}\}_{k\in I})^ {\perp} \,  \bot \, ({\overline{span}\{f_{k}\}_{k\in I^c})^ \bot}. $$
\end{theorem}
One direction of this implication holds for fusion frames.

\begin{theorem}  \label{T: Fusion frames property 1}
Let$\{W_i,v_i\}_{i\in I}$ be a fusion frame in $R^n$. If $\{W_i,v_i\}_{i\in I}$ does norm retrieval, then whenever $J\subset I$,
and $x\perp W_j$ for all $j\in J$ and $y\perp W_j$ for all $j\in J^c$, then $x\perp y$.
\end{theorem}

\begin{proof}
By our assumption,
\[ \|P_i(x+y)\|=\|P_iy\|=\|P_i(x-y)\|,\mbox{ for all }i\in J,\]
and
\[\|P_i(x+y)\|=\|P_ix\|=\|P_i(x-y)\|,\mbox{ for all }i\in J^c.\]
Since the fusion frame does norm retrieval,
\[ \|x+y\|^2=\|x\|^2+\|y\|^2+2\langle x,y\rangle=\|x-y\|^2=\|x\|^2+\|y\|^2-2\langle x,y\rangle.\]
So $\langle x,y\rangle=0$ and $x\perp y$.
\end{proof}

In contrast to the vector case, the converse of the above theorem fails in general.
\vskip12pt
\begin{example}
In $R^3$ let $\{e_i\}_{i=1}^3$ be the canonical basis and let $W_1=[e_2,e_3]$ and $W_2=[e_1,e_3]$.  If $x\perp W_1$ then
$x=ae_1$ and if $y\perp W_2$ then $y=be_2$,and so $x\perp y$. The other possibility is $x\perp (W_1\cup W_2)$, but in this
case $x=0$ so $x\perp y$ for all $y$.

Let
\[ x=2e_1+2e_2+e_3\mbox{ and }y=e_1+e_2+2e_3.\]
Then
\[ \|P_1x\|^2=2^2+1^2=5\mbox{ and }\|P_1y\|^2=1^2+2^2=5.\]
And
\[ \|P_2x\|^2=2^2+1^2=5\mbox { and } \|P_2y\|^2=1^2+2^2=5.\]
But,
\[ \|x\|^2=9\mbox{ and }\|y\|^2=6.\]
So this fusion frame fails norm retrieval.
\end{example}
\begin{theorem}\label{T: Fusion frames property 2}
If $\{P_i,v_i\}_{i=1}^m$ is an A-tight fusion frame, then it does norm retrieval.
\end{theorem}

\begin{proof}
If $\|P_ix\|=\|P_iy\|$ for all $i=1,2,\ldots,m$, then
\[ A\|x\|^2=\sum_{i=1}^mv_i\|P_ix\|^2=\sum_{i=1}^mv_i\|P_iy\|^2=A\|y\|^2.\]
So we have norm retrieval.
\end{proof}

\begin{theorem}\label{T13}
Let $\{e_i\}_{i=1}^n$ be the canonical orthonormal basis for $R^n$. Let $\{I_i\}_{i=1}^m$ be subsets of $[n]$. Let
\[
W_i=span\{e_j\}_{j\in I_i},\mbox{ for all }i=1,2,\ldots,m.\]
Assume there exists a natural number K and $\epsilon_i=\pm 1$ so that
\[ \sum_{i=1}^m\epsilon_iI_i=K(1,1,\ldots,1)\in R^n.\]
Then $\{W_i,v_i\}_{i=1}^m$ does norm retrieval for all $0<v_i<\infty$.
\end{theorem}

\begin{proof}
Let $\{P_i\}_{i\in I}$ be the projections onto $\{W_i\}_{i\in I}$. For $x\in R^n$ we have
\[ \sum_{i=1}^m\|P_ix\|^2=\sum_{i=1}^m\epsilon_i\sum_{j\in I_i}a_j^2=\sum_{j=1}^n\sum\{\epsilon_ia_j^2:j\in I_i\}=\sum_{j=1}^nKa_j^2=K\|x\|^2.
\]
It follows that the fusion frame does norm retrieval.
\end{proof}

The converse to the above theorem fails. 
Let $\{e_i\}_{i=1}^4$ be an orthonormal basis for $\RR^4$ and let 
\[ W_1=[e_1,e_4],\ W_2=[e_2,e_4],\ W_3=[e_3,e_4],\ W_4=[e_4].\]
This clearly does norm retrieval. Also, 
\[I_1=\{1,4\},\ I_2=\{2,4\},\ I_3=\{3,4\},\ I_4=\{4\}.\]
If $I\subset [4]$, $\epsilon_i=\pm 1$ for $i\in I$, and $\sum_{i\in I}\epsilon_iI_i=K(1,1,1,1)$, then $K=1$ and $\epsilon_i=1$ for $i=1,2,3$.
Since $\sum_{i=1}^3I_i=(1,1,1,3)$, this set fails the assumption in the theorem.

\section{Hyperplanes}\label{s: Hyperplanes}

The following appeared in \cite{CGJT17}.
\begin{theorem}\label{T7}
If $\{W_i\}_{i=1}^m$ are hyperplanes doing norm retrieval in $R^n$ and $\{W_i^{\perp}\}_{i=1}^m$ are linearly independent, then $m\ge n$.
\end{theorem}
The following appeared in \cite{CC17}.

\begin{theorem}\label{T5}
Let $\{W_i\}_{i=1}^m$ be subspaces of $R^n$ with projections $\{P_i\}_{i=1}^m$.  The following are equivalent:
\begin{enumerate}
\item $\{W_i\}_{i=1}^m$ does norm retrieval.
\item For every $0\not= x\in R^n$, $x\in span\{P_ix\}_{i=1}^m$.
\end{enumerate}
\end{theorem}
The following example appears in \cite{CGJT17}. We will give a new proof which generalizes to answer another problem.
\begin{example}\label{Ex1}
If $\{W_i\}_{i=1}^n$ are hyperplanes doing norm retrieval, this does not imply that $\{W_i^{\perp}\}_{i=1}^n$ are independent.
Let $\{e_i\}_{i=1}^3$ be a orthonormal basis of $R^3$ and let $\phi_1=e_1$, $\phi_2=e_2$, and $\phi_3=(e_1-e_2)/\sqrt{2}$.
The vectors are not independent but
Note that
\[ \phi_1^{\perp}=W_1=[e_2,e_3],\ \ \phi_2^{\perp}=W_2=[e_1,e_3],\ \ \phi_3^{\perp}=W_3=[(e_1+e_2)/\sqrt{2},e_3].\]
The $\{W_i^{\perp}\}_{i=1}^3$ are hyperplanes doing norm retrieval.
To see this let $\{P_i\}_{i=1}^3$ be the projections onto
$\{W_i\}_{i=1}^3$ and we will check Theorem \ref{T5}. Then if $x=(a,b,c)\in R^3$,
\begin{equation}\label{E1}
P_1x=(0,b,c)
\end{equation}
\begin{equation}\label{E2}
P_2x=(a,0,c)
\end{equation}
\begin{equation}\label{E3}
P_3x=\left ( \frac{a+b}{2},\frac{a+b}{2},c\right )
\end{equation}
Then,
\[ (3.2)-(3.1)=(a,-b,0), \: 2(3.3)=(a+b,a+b,2c).
\]
Also,
\[ 2(3.3)-(3.1)-(3.2)=(b,a,0).
\]
We leave it to the reader to check the special cases where some of $a,b,c$ are zero. For example, if $a=0\not=b,c$ then
$P_2x=(0,0,c)$ and so $e_3\in span \{P_ix\}_{i=1}^3$ and now by (3.2), $e_2\in span \{P_ix\}_{i=1}^3$ and hence
$x\in span\{e_i\}_{i=2}^3=span\{P_ix\}_{i=1}^3$. So we assume $a,b,c\not= 0$.
Since $(a,-b,0)\perp (b,a,0)$ and both are in $span\{P_ix\}_{i=1}^3$, it follows that $e_1,e_2\in span\{P_ix\}_{i=1}^3$.
Combining this with (3.3) puts $e_3\in span\{P_ix\}_{i=1}^3$. So $x\in R^3=span\{P_ix\}_{i=1}^3$.
\end{example}

\begin{example}
Theorem \ref{T7} may fail if the $\{W_i^{\perp}\}_{i=1}^n$ are not linearly independent.  Let $\{e_i\}_{i=1}^4$ be an orthonormal
basis for $\mathbb{R}^4$  and define hyperplanes
\[ W_1=[e_2,e_3,e_4],\ \ W_2= [e_1,e_3,e_4],\ \ W_3=[\frac{e_1+e_2}{\sqrt{2}},e_3,e_4],\]
and let $\{P_i\}_{i=1}^3$ be the corresponding projections.
Then
\[ W_1^{\perp}=[e_1],\ \ W_2^{\perp}=[e_2],\ \ W_3^{\perp}=[\frac{e_1-e_2}{\sqrt{2}}],\]
and $\{W_i^{\perp}\}_{i=1}^3$ are not independent. We will show that $\{W_i\}_{i=1}^3$ does norm retrieval in $R^4$. We mimic Example \ref{Ex1}.
Let $x=(a,b,c,d)$ and we check Theorem \ref{T5}.
Again we leave it to the reader to check the simple cases where one or more of the $a,b,c,d$ are zero. Using exactly the same
argument as Example \ref{Ex1}, we discover that $e_1,e_2\in span \{P_ix\}_{i=1}^3$. so $(a,0,0,0)\in span\{P_ix\}_{i=1}^3$.
Also, $P_1x=(0,b,c,d)\in span \{P_ix\}_{i=1}^3$. It follows that $x\in span\{P_ix\}_{i=1}^3$.
\end{example}

An examination of the above two examples shows a general result.

\begin{theorem}
Let $\{e_i\}_{i=1}^n$ be an orthonormal basis for $R^n$. If $\{W_i\}_{i=1}^m$ does norm retrieval for $[e_1,e_2,\ldots,e_k]$, $1\le k < n$,
then $\{W_i\cup[e_{k+1},e_{k+2},\ldots,e_n]\}_{i=1}^m$ does norm retrieval in $R^n$. If the $W_i$ are hyperplanes in $[e_1,\ldots,e_k]$
then the new sets are hyperplanes in $R^n$.
\end{theorem}

\section{Full spark, Finitely full spark and norm retrievability}
{\bf Spark} is also an essential issue here.
\begin{definition}\label{D:full spark}
\textrm A family of vectors $\{f_{i}\}_{i=1}^{m}$ \textrm in $\mathbb{R}^n \, (m \geq n)$ has {\bf spark k} if for every $I \subset [m] \, \mbox with \, |I|=k-1$ \textrm, $\{f_{i}\}_{i \in I}$ \textrm
is linearly independent. It is full spark if $k=n+1$ and hence every n-element subset spans  $\mathbb{R}^n$.
\end{definition}

It is proved in \cite{BCCHT18} that finitely full spark frames are dense in all frames in both the finite and infinite dimensional case.
 By theorem \ref{T:phase retrievality and complement property for finite case}\, the families of vectors $\{f_{k}\}_{k=1}^{m}$ which do phase retrieval in $\mathbb{R}^n$ are dense in the family of $m\geq
 (2n-1)$-element sets of vectors in $\mathbb{R}^n$ for every finite $n$. Since every phase retrievable frame in $\mathbb{R}^n$ is a norm retrievable frame we have:
 the families of vectors $\{f_{k}\}_{k=1}^{m}$ which do norm retrieval in $\mathbb{R}^n$ are dense in the family of $m\geq (2n-1)$-element sets of vectors in $\mathbb{R}^n$ for every finite $n$.
 Now we focus on the case when $m\leq (2n-2)$.
We will show that
the families of norm retrievable frames $\{f_{k}\}_{k=1}^{m}$ in $\mathbb{R}^n$ are not dense in the family of $m\leq (2n-2)$-element sets of vectors in $\mathbb{R}^n$ for every finite $n$. This will require some preliminary results.
\begin{theorem} \label{T:equivalent cases}
Let $X$ and $Y$ be subspaces of $\mathbb{R}^n$ and $T:X \longrightarrow Y$ be an operator with $\|I-T\|< \epsilon$. Let $P$ be the projection onto $X$. Define $S:\mathbb{R}^n \longrightarrow \mathbb{R}^n$ by:
$$S(Px+(I-P)x)=TPx+(I-P)x.$$
Then
\begin{enumerate}
\item \quad $\|I-S\|< \epsilon$ \ and \ $\|S\|< 1+ \epsilon$ \\
\item \quad $\|S^{-1}\|< \frac{1}{1- \epsilon}$ \ and \ $\|I-S^{-1}\|< \frac{\epsilon}{1- \epsilon}$ \\
\item \quad $Q=SPS^{-1}$ \ is \ a \ projection \ onto \ Y. \\
\item \quad  $\|P-Q\|< \frac{1+ \epsilon^2}{1- \epsilon}$ \\
\item \quad We have for $x\in X$, $\|Qx\|< \epsilon \|x\|$. \\
\item \quad For  $x \perp X$, $\|(I-Q)x\| \geq (1- \epsilon) \|x\|.$ \\
\item \quad For  $x \perp X$, $\|x\|=1$,
$\|x-\frac{(I-Q)x}{\|(I-Q)x\|}\| \leq \frac{2 \epsilon}{1- \epsilon}= \delta$.
\end{enumerate}
It follows that given $\delta>0$ there is an $\epsilon >0$ so that given the assumptions of the theorem, if $x\perp X$ with $\|x\|=1$, there is a
$y\perp Y$ with $\|y\|=1$ and $\|x-y\|<\delta$.
\end{theorem}
\begin{proof}
(1)  We compute
\begin{align*}
\|(I-S)x\|&= \|x-Sx\|\\
&= \|Px+(I-P)x-S(Px+(I-P)x)\|\\
&= \|Px+(I-P)x-TPx-(I-P)x\|\\
&=\|Px-TPx\|\\
&= \|(I-T)Px\|\\
&\le \epsilon \|Px\| \le \epsilon \|x\|.
\end{align*}
Also,
\[ \|Sx\|\le \|x\|+\|x-Sx\| \le \|x\|+\epsilon \|x\|=(1+\epsilon)\|x\|.\]
(2) By the Neuman series,
\[ S^{-1}=(I-(I-S))^{-1} = \sum_{i=0}^{\infty}(I-S)^i,\]
and so
\[ \|S^{-1}\|\le \sum_{i=0}^{\infty}\|(I-S)^i\|=1+\sum_{i=1}^{\infty}\epsilon^i= 1+\frac{\epsilon}{1-\epsilon}=\frac{1}{1-\epsilon}.\]
Also,
\[ \|I-S^{-1}\| =\| \sum_{i=1}^{\infty}(I-S)^i\|\le \sum_{i=1}^{\infty}\|(I-S)^i\|\le \frac{\epsilon}{1-\epsilon}.\]
(3) Since $SPS^{-1}SPS^{-1}= SP^2S^{-1}=SPS^{-1}$, this is a projection. If $y\in Y$, $S^{-1}y\in X$ and so $PS^{-1}y=S^{-1}y$ and
hence $SPS^{-1}SPS^{-1}y=SS^{-1}y=y$. \\[10pt]
(4) We compute
\begin{align*}
\|P-Q\|&=\|P-SPS^{-1}\|\\
&=\|P-SP+SP-SPS^{-1}\|\\
&\le \|(I-S)P\|+\|SP(I-S^{-1})\|\\
&\le \|I-S\|\|P\|+\|SP\|\|I-S^{-1}\|\\
&\le \epsilon +\frac{\epsilon}{1-\epsilon}\|S\|\|P\|\\
&\le 1+\frac{\epsilon(1+\epsilon)}{1-\epsilon}=\frac{1+\epsilon^2}{1-\epsilon}.
\end{align*}
(5) We compute for $x\perp X$,
\[ \|Qx\|\le \|(P-Q)x\|+\|Px\| \le (\epsilon +0)\|x\|\]
(6) We compute for $x\perp X$
\[ \|(I-Q)x\|\ge \|x\|-\|Qx\|\ge \|x\|-\|Qx\|\ge (1-\epsilon)\|x\|.\]
(7) We compute for $x\perp X$, $\|x\|=1$ using (5) and (6),
\[1-\epsilon \le \|(I-Q)x\|\le \|x\|+\|Qx\|\le 1+\epsilon.\]
So
\[\frac{1}{1+\epsilon} \le \frac{1}{\|(I-Q)x\|}\le \frac{1}{1-\epsilon},\]
so
\[\frac{-1}{1+\epsilon}\ge \frac{-1}{\|(I-Q)x\|}\ge \frac{-1}{1-\epsilon},\]
and
\[ 1-\frac{1}{1+\epsilon}=\frac{\epsilon}{1+\epsilon}\ge 1-\frac{1}{\|(I-Q)x\|} \ge 1-\frac{1}{1-\epsilon}=\frac{-\epsilon}{1-\epsilon}.\]
Hence,
\[ \left| 1-\frac{1}{\|(I-Q)x\|}\right | \le \frac{\epsilon}{1-\epsilon}.\]
Now,
\begin{align*}
\left \|x-\frac{(I-Q)x}{\|(I-Q)x\|}\right \|&\le \|x-(I-Q)x\|+\left \|(I-Q)x-\frac{(I-Q)x}{\|(I-Q)x\|}\right \|\\
&= \|Qx\|+\|(I-Q)x\|\left |1-\frac{1}{\|(I-Q)x\|}\right |\\
&\le \epsilon +(1+\epsilon)\frac{\epsilon}{1-\epsilon}=\epsilon \left ( 1+\frac{1}{1-\epsilon}+\frac{\epsilon}{1- \epsilon} \right )= \frac{2\epsilon}{1-\epsilon}.
\end{align*}
\end{proof}
We can now prove the main result.

\begin{theorem} \label{T:norm retrievality and perturbation}
Let \ $\{x_i\}_{i=1}^m$, $m\leq 2n-2$ be a frame in $\mathbb{R}^n$ which fails norm retrieval. Then there is an $\epsilon>0$ so that whenever
$\{y_i\}_{i=1}^m$ are vectors satisfying
\[ \sum_{i=1}^m\|x_i-y_i\|<\epsilon,\]
then $\{y_i\}_{i=1}^m$ also fails norm retrieval.
\end{theorem}
\begin{proof}
Since $\{x_i\}_{i=1}^m$ fails norm retrieval, there is some $I\subset [m]$ so that there are vectors $\|x\|=\|y\|=1$, $x\perp x_i$ for all $i\in I$,
$y\perp x_i$ for all $i\in I^c$ and $x$ is not orthogonal to $y$. That is, there is a $\delta>0$ so that $\|x-y\|\le \sqrt{2}-\delta$. Choose $\epsilon >0$
so that
\[ \frac{2\epsilon}{1-\epsilon}<\frac{\delta}{3}.\]
Let $J_1\subset I$ and $J_2\subset I^c$ with $\{x_i\}_{i\in J_k}$ linearly independent for $k=1,2$ and
\[ X_1=span\{x_i\}_{i\in J_1}=span\{x_i\}_{i\in I}\mbox{ and } X_2=span\{x_i\}_{i\in J_2}=span\{x_i\}_{i\in I^c}.\]
Let $Y_1=span\{y_i\}_{i\in J_1}$ and $Y_2=span\{y_i\}_{i\in J_2}$ and define operators $T_k:X_k\rightarrow Y_k$
for $k=1,2$ by $T_kx_i=y_i$ for $i\in J_k$.
By Theorem \ref{T:equivalent cases} there are vectors $\|z\|=\|w\|=1$ with $z\perp y_i$ for all $i\in I$, $w\perp y_i$ for all $i\in I^c$ and
\[ \|x-z\|<\frac{\delta}{3}\mbox{ and } \|y-w\|<\frac{\delta}{3}.\]
Now
\[ \|z-w\|\le \|z-x\|+\|x-y\|+\|y-w\|< \frac{\delta}{3}+\sqrt{2}-\delta+\frac{\delta}{3}=\sqrt{2}-\frac{\delta}{3}.\]
It follows that $z$ and $w$ are not orthogonal and so by Theorem \ref{T:norm retrievablity and perp} $\{y_i\}_{i=1}^m$ fails norm retrieval.
\end{proof}
On the surface, it looks like the above argument works equally well for $m\ge 2n-1$. The problem is that in this case, if the vectors are full
spark then whenever we divide them into two sets, one will span the whole space. So the only vector orthogonal to this set is the zero
vector and this vector is orthogonal to all vectors. I.e. This set does norm retrieval.

Now, we will consider the infinite dimensional Hilbert space $\ell ^2$. In $\ell ^2$, every phase (norm) retrievable family need not to be full spark. For example if we write a norm retrievable set twice, the latter is a
norm retrievable set again, but it is not full spark. The corresponding definition of a full spark family for $\ell ^2$ first appeared in  \cite{BCCHT18}.
\begin{definition}\label{D:family full spark in}
A family of vectors $\{f_{k}\}_{k=1}^{\infty}$ in $\ell^ 2$ is \textbf{full spark} if every infinite subset spans $\ell ^2$.
\end{definition}
If a set is full spark in  $\ell^{2}$, we can drop infinitely many vectors and as long as there are infinitely many left, it still spans  $\ell^{2}$.

The concept of ``finitely full spark vectors for $\ell^2$" first was introduced in \cite{BCCHT18}.\\
\begin{definition}\label{D:finitely full spark}
A set of vectors $\{f_{k}\}_{k=1}^{\infty}$ in $\ell^{2}$ is \textbf{finitely full spark} if for every $I \subset \mathbb{N}$ with $|I|=n$,  $\{P_{I}f_{k}\}_{k=1}^{\infty}$ is full spark (i.e., spark
$n+1$), where $P_{I}$ is the orthogonal projection onto $span\{e_{k}\}_{k \in I}$.
\end{definition}

\begin{example}
The  set $\{ 1/2^i e_1+1/{\{2^i+1}\} e_2 + \dots + 1/{\{2^i+(n-1)}\}e_n\}_{i=1}^{\infty}$
is finitely linear independent and so is finitely full spark for $\ell ^2$ for any arbitrary $n\in \mathbb{N}$.
\end{example}

It is shown in \cite{BCCHT18} that the families of vectors which do phase retrieval in $\ell^2$ are not dense in the infinite families of vectors in $\ell^2$. We will show a similar result for the families
of vectors which do norm retrieval in $\ell^2$.
By \cite {CCJW14} we know if $\{\phi_i\}_{i=1}^{n}$ does norm retrieval in $\mathbb{R}^{n}$, then the vectors of the frame are orthogonal. This is true in  $\ell^2$ also.
\begin{proposition}\label{T:every Riesz bases that do norm retrieval are orthogonal sets}
Let $\{\phi_i\}_{i=1}^{\infty}$ be a Riesz basis doing norm retrieval in $\ell^2$, then the vectors $\{\phi_i\}_{i=1}^{\infty}$ are orthogonal.
\end{proposition}
\begin{proof}
 Without loss of generality, assume $\|\phi_i\|=1$ and there is some $j\in I$ with
$\phi_j$ not orthogonal to span$\{\phi_i\}_{i \ne j}$. Choose a unit vector $x \perp \phi_i$ for all
$i \ne j$ So that $x\not= \phi_j$. Let $y=x-\langle x,\phi_j\rangle \phi_j$. Now
$\langle \phi_j,y\rangle=\langle \phi_j,x\rangle-\langle x,\phi_j\rangle \langle \phi_j, \phi_j \rangle=0$. Let $I=\{i: i \ne j \}$. Then $x \perp span \{\phi_i\}_{i \in I}$ and $y \perp \phi_j$, but
$$ \langle x,y \rangle=\langle x,x \rangle - \langle x,\phi_j \rangle \langle x,\phi_j \rangle=
1-|\langle x, \phi_j \rangle |^2 \ne 0,$$
Contradicting Theorem \ref{T:norm retrievablity and perp}.
\end{proof}
\begin{lemma} \label{T:unit norm sequences}
Let $\{\phi_i\}_{i=1}^{\infty}$ and $\{\psi_i\}_{i=1}^{\infty}$ be Riesz bases for $\ell^2$. Given $\epsilon >0$ arbitrary, there exists $\delta >0$ such that
$$\mbox{if }\sum_{i=1}^{\infty} \|\phi_i-\psi_i\| \leq \delta\quad \mbox{ then }\quad
\sum_{i=1}^{\infty} \left \|\frac{\phi_i}{\|\phi_i\|}-\frac{\psi_i}{\|\psi_i\|}\right \| \leq \epsilon$$.
\end{lemma}
\begin{proof} Since $\{\phi_i\}_{i=1}^{\infty}$ and $\{\psi_i\}_{i=1}^{\infty}$ are Riesz bases,
there are constants $0<A\leq B< \infty$ satisfying $A \leq \|\phi_i\| \leq B$ and $A \leq \| \psi_i \| \leq B$, for all $i \in \mathbb{N}$.
Assume $\sum_{i=1}^{\infty} \|\phi_i-\psi_i\| \leq \delta$ with $ \frac{2B}{A^2} \delta= \epsilon$, then we have
\begin{align}
\sum_{i=1}^{\infty} \left \|\frac{\phi_i}{\|\phi_i\|}-\frac{\psi_i}{\| \psi_i\|}\right \| \notag
&= \sum_{i=1}^{\infty} \frac{1}{\|\phi_i\|\|\psi_i\|}\| \|\psi_i \|\phi_i-\|\phi_i\| \psi_i \| \notag \\
&\leq \frac{1}{A^2}\sum_{i=1}^{\infty}\|\|\psi_i \|\phi_i-\|\phi_i\| \phi_i+\|\phi_i\| \phi_i-\|\phi_i\| \psi_i \| \notag \\
&\leq \frac{1}{A^2}\sum_{i=1}^{\infty}[\|\|\psi_i \|\phi_i-\|\phi_i\| \phi_i\|+\| \|\phi_i\| \phi_i-\|\phi_i\| \psi_i \|] \notag \\
&\leq \frac{1}{A^2}\sum_{i=1}^{\infty} |\|\psi_i\|-\|\phi_i\|| \|\phi_i\|+ \frac{1}{A^2}\sum_{i=1}^{\infty}
\|\phi_i\|\|\phi_i-\psi_i\|     \notag \\
&\leq \frac{B}{A^2}\sum_{i=1}^{\infty} \|\psi_i-\phi_i\| + \frac{B}{A^2}\sum_{i=1}^{\infty}
\|\psi_i-\phi_i\| \notag \\
&= \frac{2B}{A^2}\delta= \epsilon \notag
\end{align}
\end{proof}
\begin{theorem}\label{T:Norm retrievable families are not dense }
The families of vectors which do norm retrieval in $\ell^2$ are not dense in the infinite families of vectors in $\ell^2$.
\end{theorem}
\begin{proof}
  We prove the Theorem in three steps. \\
Step 1: By Proposition \ref{T:every Riesz bases that do norm retrieval are orthogonal sets} if $\{\phi_i\}_{i=1}^{\infty}$ be a Riesz basis doing norm retrieval in $\ell^2$, then the vectors of
$\{\phi_i\}_{i=1}^{\infty}$ are orthogonal.\\
Step 2: It is known that if $\{\phi_i\}_{i=1}^{\infty}$ is a Riesz basis and $\{\psi_i\}_{i=1}^{\infty}$ is close enough to it, then  $\{\psi_i\}_{i=1}^{\infty}$ must be a Riesz basis. \\
Step 3: If $\{\phi_i\}_{i=1}^{\infty}$ is a Riesz basis and $\{\psi_i\}_{i=1}^{\infty}$ is an orthogonal set of vectors arbitrary close enough to $\{\phi_i\}_{i=1}^{\infty}$, then
$\{\phi_i\}_{i=1}^{\infty}$ must be orthogonal.
To prove it: \\
We know  by the parallelogram law if $\|x\|=\|y\|=1$, then $\|x-y\|=\sqrt{2}$ if and only if $x\perp y$.
 In the above problem we may assume that for simplicity we only have two vectors $\{x_1,x_2\} \subset
\{\phi_i\}_{i=1}^{\infty}$ with $\|x_1\|=\|x_2\|=1$ and given $\epsilon >0$
by Lemma \ref{T:unit norm sequences}
there are vectors $\{y_1,y_2\} \subset \{\psi_i\}_{i=1}^{\infty}$ such that $\|y_1\|=\|y_2\|=1$, $y_1 \perp y_2$ and $\|x_i-y_i\|< \epsilon$ then
\[ \sqrt{2}-2\epsilon = \|y_1-y_2\|-2\epsilon \le \|x_1-x_2\|\le \|y_1-y_2\|+2\epsilon= \sqrt{2}+2\epsilon.\] \\
Since $\epsilon$ was arbitrary, it follows that $\sqrt{2}= \|x_1-x_2\|$ and $\{\phi_i\}_{i=1}^{\infty}$ is an orthogonal sequence. Therefore if $\{\phi_i\}_{i=1}^{\infty}$ is a non-orthogonal Riesz basis,
 there is no norm retrievable sequence $\{\psi_i\}_{i=1}^{\infty}$ that is orthogonal and close enough to $\{\phi_i\}_{i=1}^{\infty}$. Thus the families of vectors which do norm retrieval in
$\ell^2$ are not dense in the infinite families of vectors in $\ell^2$.
\end{proof}
Of course by Theorem 8, we can construct many none orthogonal Riesz basic sequences that
are close to an orthogonal Riesz basic sequence.

\begin{example}
 Let $0< \epsilon <1$ be given, let $\{e_i\}_{i=1}^{\infty}$ be an orthonormal basis for $\ell^2$, and choose any vectors $\{x_i\}_{i=1}^{\infty}$
 so that
 \[ \sum_{i=1}^{\infty}\|x_i\|< \epsilon.
 \]
 Then $\{e_i+x_i\}_{i=1}^{\infty}$ is a Riesz basis failing norm retrieval and it is $\epsilon$-close to an orthonormal basis.
 \end{example}

\end{document}